\documentclass[12pt,a4paper]{article}

\newif\ifpdf
\ifx\pdfoutput\undefined
\pdffalse 
\else
\pdfoutput=1 
\pdftrue \fi

\ifpdf
\usepackage[pdftex]{graphicx}
\else
\usepackage{graphicx}
\fi

\usepackage{amssymb}
\usepackage{rotating}
\usepackage{amsmath}
\usepackage[pdftex]{graphicx}
\usepackage{epsfig}
\usepackage[pdftex]{color}

\newcommand{\fhat}[1]{\widehat{#1}}
\newcommand{\Z}{\mathbb{Z}}

\newcommand{\R}{\mathbb{R}}

\newcommand{\N}{\mathbb{N}}

\newcommand{\C}{\mathbb{C}}

\newcommand{\calS}{{\cal S}}

\newcommand{\calF}{{\cal F}}

\newcommand{\calM}{{\cal M}}

\newcommand{\gauss}{\mathfrak{g}}

\newcommand{\LtR}{L^2(\mathbb{R}^d)}

\newtheorem{theorem}{Theorem}
\newtheorem{lemma}[theorem]{Lemma}
\newtheorem{proposition}[theorem]{Proposition}
\newtheorem{definition}[theorem]{Definition}

\newtheorem{corollary}[theorem]{Corollary}

\ifpdf \DeclareGraphicsExtensions{.pdf, .jpg} \else
\DeclareGraphicsExtensions{.eps, .jpg} \fi

\title{Remarks on multivariate Gaussian Gabor frames}
\author{G\"otz E. Pfander, Peter Rashkov}
\date{August 20, 2010}

\begin{document}

\maketitle

\begin{abstract}
We report on initial findings on Gabor systems with multivariate Gaussian
window. Unlike the existing characterisation for dimension one in terms of
lattice density, our results indicate that the behavior of Gaussians  in
higher-dimensional Gabor systems is intricate and further exploration is a
valuable and challenging task.
\end{abstract}

\section{Introduction}
Gabor's seminal paper~\cite{Gab46} claimed that every function in
$L^2(\R)$ can be well represented as a series of translated and modulated
copies of the  Gaussian $\gauss_1(x)=e^{-\pi|x|^2}$. In detail, he suggested
that for every $f\in L^2(\R)$ there exists a  sequence of scalars $c_{kl}(f)$
such that \begin{equation}\label{eq:gse} f(t)=\sum_{k,l\in\Z}c_{kl}(f)e^{2\pi i lt}
\gauss_1(t-k).
\end{equation}
But while its central role in analysis and its wide spectrum of nice analytic
properties, for example, optimal time--frequency concentration, make the
Gaussian a natural candidate to be a window function for so-called Gabor
systems, it is now well established that any series representation of the form
\eqref{eq:gse} only converges to $f$ in a distributional sense, and not in the
$L^2$-norm \cite{J81}. Today, the spanning properties of the Gabor system
$(\gauss_1,a\Z\times b\Z)=\{e^{2\pi iblt}\gauss_1(t-ak)\}$, $a,b>0$, are fully
understood, for example, the system suggested by Gabor turns out to be
 overcomplete (\cite{Dau90,Lyu92,SW92} and Theorem~\ref{th:lyub} below).

Multivariate Gaussian and general Gabor systems though are far from being
 understood. While, for example, it is known that for any $g$, the Gabor system
$(g,\Lambda)=\{e^{2\pi i\omega t}g(t-x):(x,\omega)\in\Lambda\}$ is not a
frame for $\LtR$ if the set $\Lambda$ has density less than 1, nontrivial
positive results for  given window functions such as the Gaussian are scarce
in the literature~\cite{H07}. Simultaneously to our work, Gr\"ochenig has
started to study of multivariate Gaussian Gabor systems. His focus though is
on so-called complex lattices~\cite{GroGa}. Here, we provide some results
that illustrate the intricate structure of Gaussian Gabor frames in higher
dimensions for real lattices.

\section{Gabor frames}
We denote by $\gauss_d$ the $d$-dimensional normalized Gaussian $2^{\frac d4}e^{-\pi\|x\|^2}$. It is clear that
$\gauss_d=\displaystyle\otimes_{d\,\text{times}}\,\gauss_1$. 
A
\emph{translation} or \emph{time shift} is the operator $(T_x
f)(t) = f(t-x),\,x\in \R^d$, and a \emph{modulation} or
\emph{frequency shift} is the operator $(M_\omega f)(t)= e^{2\pi
i\langle\omega,t\rangle}f(t),\,\omega\in\fhat \R^d$. A
\emph{time-frequency shift} is then
\[
(\pi(\lambda)f)(t)=(M_\omega T_x f)(t)=e^{2\pi i\langle
\omega,t\rangle}f(t-x),\quad \lambda=(x,\omega)\in\R^d\times\fhat\R^d.
\]
\begin{definition}\label{def:Gaborsys-d}
Let $\Lambda\subset\R^{2d}$ be a discrete set. A Gabor system $(g,
\Lambda)$ for $L^2(\R^d)$ is the set of all time-frequency shifts of the
window function $g$ by $\lambda=(x,\omega)\in\Lambda$, that is,
$(g,\Lambda)=\{\pi(\lambda)g: \lambda\in\Lambda\}$.
\end{definition}

Note that it is an easy consequence from Fourier analysis that
$(\chi_{[0,1)},\Z\times\Z)$ is an orthonormal basis for $L^2(\R)$. 
For $(\chi_{[0,1)},\Lambda)$, with $\Lambda\neq\Z\times\Z$, the situation is
already  quite delicate as shown in~\cite{Jans}.

While Gabor orthonormal bases are useful, the so-called Balian-Low
Theorem implies that they have a crucial shortcoming. Namely, if
$(g,\Lambda)$ is an  orthonormal basis for $L^2(\R)$, then $g$ can not be
well-localized in both time and frequency, in fact, we then have
\[\int|xg(x)|^2\,dx\int|\xi\fhat g(\xi)|^2\,d\xi=\infty.\]
Consequently, in Gabor analysis we resort to consider frames and Riesz bases.
\begin{definition}
A family of functions $\{\phi_k\}_{k\in\Z}\subset L^2(\R^d)$ with
\begin{equation}\label{eq:framedef}
    A\|f\|_2^2\, \leq\,  \sum_k |\langle f,\phi_k \rangle|^2 \,
    \leq \,
    B\|f\|_2^2\, , \quad f\in L^2(\R^d),
\end{equation}
for positive $A$ and $B$ is called a frame for $L^2(\R^d)$. The constants
$A$ and $B$ are called, respectively, a lower frame bound  and  an upper
frame bound of the frame $\{\phi_k\}$.
\end{definition}
\begin{definition}
A family of functions $\{e_k\}_{k\in\Z}\subset L^2(\R^d)$ with
\begin{equation}\label{eq:rbdef} \notag
    A\|c\|_2^2\, \leq\,  \|\sum_k c_ke_k \|_2^2 \,
    \leq \,
    B\|c\|_2^2\, , \quad c\in \ell^2_0(\R^d),
\end{equation}
for positive $A$ and $B$ is called a Riesz basis.
\end{definition}
For a detailed description of frame and Riesz basis theory we refer to
\cite{G01,C03}.

To consider lattices, or more general countable sets in $\R^{2d}$, we state
the  definition of density and summarize its role in Gabor analysis. Let
$\mathfrak{B}_d(R)$ denote the $l^2$-ball in $\R^{d}$ centered at 0 and with
radius $R$.
\begin{definition}\label{def:beurling}
The lower and
upper Beurling densities of $\calM\subset\R^d$ are given by, respectively,
\begin{equation}\label{eqdef:beurling} \notag
\begin{aligned}
D^-(\calM)&=\liminf_{R\to\infty}\inf_{z\in \R^d}\frac{|\calM\cap
\{\mathfrak{B}_d(R)+z\}|}{\pi R^d}, \\
D^+(\calM)&=\limsup_{R\to\infty}\sup_{z\in \R^d}\frac{|\calM\cap
\{\mathfrak{B}_d(R)+z\}|}{\pi R^d}.
\end{aligned}
\end{equation}
Whenever $D^-(\calM)=D^+(\calM)$, then $\calM$ is said to have a uniform
Beurling density, denoted by $D(\calM)=D^-(\calM)=D^+(\calM)$. $\calM$ is
uniformly separated if $\inf\{|\lambda-\mu|:\lambda\neq\mu\in\calM\}>0$. If
$\calM$ is a full-rank lattice, that is, $\calM=A\Z^{d},\det A\neq0$, then
$\calM$ is uniformly separated and  $D(\calM)=\frac1{\det A}$.
\end{definition}
\begin{theorem}[Density theorem]\label{thm:density} Let $g\in L^2(\R^d)$   and let $\Lambda$ be a full-rank lattice.
\begin{enumerate}
\item If $D(\Lambda)<1$, then $(g,\Lambda)$ is incomplete in
    $L^2(\R^d)$.
\item If $(g,\Lambda)$ is a frame for $L^2(\R^d)$, then
$D(\Lambda)\ge1$.
\item If $(g,\Lambda)$ is a Riesz basis for its closed linear span, then
$D(\Lambda)\le1$.
\end{enumerate}
Thus, if $(g,\Lambda)$ is an orthonormal basis, then $D(\Lambda)=1$.
\end{theorem}
The results listed in Theorem~\ref{thm:density} have roots in various papers;
 they are nicely summarized in~\cite{H07}.
In the one-dimensional case, Gaussian Gabor frames and Riesz bases are
 well characterized~\cite{Lyu92,SW92}.
\begin{theorem}\label{th:lyub}
Let $\Lambda\subset\R^2$. The Gabor system $(\gauss_1,\Lambda)$ is a
frame if and only if there exists a uniformly separated
$\Lambda'\subset\Lambda$ such  that $1<D^-(\Lambda')\le
D^+(\Lambda)<\infty$. If $\Lambda$ is uniformly separated and
$D^+(\Lambda)<1$, then $(\gauss_1,\Lambda)$ is a Riesz basis of a proper
subspace of $L^2(\R)$.
\end{theorem}
In their proofs, Lyubarski and Seip-Wallsten used methods from complex
analysis; the connection between Gaussian Gabor frames and complex
analysis is described below.

The \emph{short-time Fourier transform}, also called \emph{continuous
Gabor transform} or \emph{windowed Fourier transform} is defined formally
by
\begin{equation}\label{def:STFT}\notag
V_g f(x, \omega) = \int_{\R^d} f(t) e^{-2 \pi i \omega
t}\overline{g(t-x)} dt = \langle f, M_\omega T_x g\rangle = \langle
\fhat f, T_\omega M_{-x} \fhat g\rangle
\end{equation}
If $\|g\|_2=1$, for example, if $g=\gauss_d$, then  the short-time Fourier
transform is a unitary operator, so
$\|V_gf\|_{L^2(\R^{2d})}^2=\|f\|^2$. 
In the Gabor case, the frame property~\eqref{eq:framedef}   becomes a
sampling set condition: there exist constants $A,B>0$ such that
\begin{equation}\label{eq:equiv} \notag
A\|V_g
f\|_{L^2(\R^{2d})}^2\le\sum_{\lambda\in\Lambda}|V_g f(\lambda)|^2\le B\|V_g
f\|_{L^2(\R^{2d})}^2,\quad f\in \LtR.
\end{equation}

The Bargmann-Fock space is defined by
\[\calF(\C^d)=\left\{F-\text{entire with }\|F\|_\calF=\left(\int_{\C^d}|F(z)|^2e^{-\pi|z|^2}dz\right)^{\frac12}<\infty\right\}.\]
The Bargmann transform $B$ maps $L^2(\R^d)$ unitarily onto
$\calF(\C^d)$ by
\begin{equation}\label{eq:bt}\notag B:f\mapsto
Bf:z\mapsto2^{\frac14}\int f(t)e^{2\pi itz-\pi
t^2-\pi\frac{z^2}2}dt,\quad z\in\C^d.\end{equation}
With this notation it is easy to see that
\begin{equation*}
V_{\gauss_d}f(x,-\xi)=e^{2\pi ix\xi}Bf(x+i\xi)e^{-\frac\pi2|x+iy|^2},
\end{equation*}
\cite{Lyu92,SW92}.  This demonstrates that $(\gauss_d,\Lambda)$ is a
frame if and only if $\Lambda$ is a sampling set for $\calF$. In $d=1$ this
was used to prove Theorem~\ref{th:lyub}~\cite{SW92,Lyu92}. But in higher
dimension it appears as if as little is known about sampling in
Bargmann-Fock spaces as is known about multivariate Gaussian Gabor
frames (see \cite{GroGa} for a more detailed discussion of this).




\section{Gaussian Gabor frames for $L^2(\R^d)$}\label{section:propositions}

The easiest way to create frames for $L^2(\R^d),d\ge2,$ is to take tensor products of lower-dimensional frame systems. 
For $n$ lattices $\Lambda_1,\ldots,\Lambda_n$ of the same dimension, we set $\odot_{i=1}^n\Lambda_i=\{(x_1,\ldots,x_n)\times(\omega_1,\ldots,\omega_n):(x_i,\omega_i)\in\Lambda_i\}$.
\begin{lemma}\label{lemma:tensor-gabor-frames}
Let $(g,\Lambda_1)$ and $(h,\Lambda_2)$ be frames for
$L^2(\R^d)$. Then $(g\otimes h,\Lambda_1\odot\Lambda_2)$ is a
frame for $L^2(\R^{2d})$.
\end{lemma}
For a simple proof we refer to~\cite{PRtr}.

\begin{proposition}\label{prop:ex1}
Let $\Lambda=\Z^2\times \left(\begin{smallmatrix}
a&0\\
0&b\end{smallmatrix}\right)\Z^2$. If $a< 1$ and $b< 1$, then $(\gauss_2,\Lambda)$ is a
frame for $L^2(\R^2)$. If $a=b=1$, $(\gauss_2,\Lambda)$ is complete in $L^2(\R^2)$, but
not a frame. If $a>1$ or $b>1$, then $(\gauss_2,\Lambda)$ is incomplete.
\end{proposition}
Note that here $D(\Lambda)=\frac1{ab}$.
 Below we shall use the Zak transform which is defined via the series
\[Zf(x,\omega)=\sum_{k\in\Z^d}f(x-k)e^{2\pi i k\omega},\quad(x,\omega)\in\R^{2d}.\]
For a detailed presentation of the properties of the Zak transform we refer
to~\cite{G01,C03}.

\textit{Proof of Proposition~\ref{prop:ex1}}. If $a< 1, b< 1$, then
Theorem~\ref{th:lyub} implies $(\gauss_1,\Z\times a\Z)$ and
$(\gauss_1,\Z\times b\Z)$ are frames for $L^2(\R)$.
Lemma~\ref{lemma:tensor-gabor-frames} implies that
$(\gauss_2,\Lambda)$ is a frame for $L^2(\R^2)$.


To show completeness of the Gabor system for $a=b=1$, we 
observe that for
$(x,\omega)=(x_1,x_2,\omega_1,\omega_2)$
\begin{equation*}
Z \gauss_2(x,\omega)=Z \gauss_1(x_1,\omega_1)\cdot Z \gauss_1(x_2,\omega_2)
\end{equation*}
Because $(\gauss_1,\Z\times\Z)$ is complete in $L^2(\R)$, but not a frame,
according to Proposition 9.4.3 in~\cite{C03}, $Z\gauss_1$ vanishes on a set
of measure zero in $[0,1)^2$. Hence, the Zak transform $Z\gauss_2$
vanishes only on a set of zero measure in $[0,1)^4$. According to
Proposition 9.4.3 in~\cite{C03}, $(\gauss_2,\Z^2\times\Z^2)$ is complete.
Furthermore, since $\gauss_2\in\calS(\R^2)$, its Zak transform is
continuous. Hence,  it is not bounded away from 0 almost everywhere.
Proposition 8.3.2 in~\cite{G01} implies that the Gabor system
$(\gauss_2,\Z^2\times \Z^2)$ is not a frame for $L^2(\R^2)$.

If $a>1$ or $b>1$, 
say $b>1$, then 
$(\gauss_1,\Z\times b\Z)$ is incomplete in $L^2(\R)$. Hence we can choose
$f_1\in L^2(\R^d),f_1\neq0$ such that $V_{\gauss_1}f_1(m_1,bm_1)=0$ for
all $(m_1,n_1)\in\Z^2$. Then for any $f_2\in L^2(\R),f_2\neq0$, the
STFT
\begin{equation}\label{eq:eqn-factoriz}
V_{\gauss_2}(f_1\otimes f_2)(m_1,m_2,an_1,bn_2)=V_{\gauss_1}(m_1,an_1)V_{\gauss_1}(m_2,b n_2)=0.
\end{equation}
But $f_1\otimes f_2\neq0$, so $(\gauss_2,\Z\times\Z\times a\Z\times b\Z)$ is incomplete.
\hfill$\square$
\begin{proposition}\label{prop:gaussian-block-diagonal}
Let $\Lambda=\Z^2\times\left(\begin{smallmatrix}
a& a\\
-b&b
\end{smallmatrix}\right)\Z^2$. Then the Gabor system $(\gauss_2,\Lambda)$ is a
frame for $L^2(\R^2)$ if  $a,b<\frac12$. If $a=b=\frac12$, then
$(\gauss_2,\Lambda)$ is complete but not a frame for $L^2(\R^2)$. If
$a,b>\frac12$, then $(\gauss_2,\Lambda)$ is incomplete.
\end{proposition}
Note that here $D(\Lambda)=\frac1{2ab}$, so $(\gauss_2,\Lambda)$ is {\it
a-priori} not a frame  if $2ab>1$.

\begin{proof}
For $\Lambda=\Z^2\times\left(\begin{smallmatrix}
a& a\\
-b&b
\end{smallmatrix}\right)\Z^2$, $F=f_1\otimes
f_2\in L^2(\R^2)$, we have
\begin{equation*}
V_{\gauss_2}F(m_1,m_2,a(n_1+n_2),b(n_2-n_1))
=V_{\gauss_1}f_1(m_1,a(n_1+n_2))\cdot V_{\gauss_1}f_2(m_2,b(n_2-n_1))
\end{equation*}
If $n_1,n_2$ are of the same parity, then $n_1\pm n_2$ is always even,
otherwise, $n_1\pm n_2$ is odd. Hence, if $a,b>\frac 12$, then we can
choose a nonzero $f_1\in L^2(\R)$ such that
$V_{\gauss_1}f_1(m_1,a(n_1+n_2))=0$, for all $m_1$ and all $(n_1, n_2)$
with $n_1-n_2$ even, and a nonzero $f_2\in L^2(\R)$ such that
$V_{\gauss_1}f_2(m_2,b(n_2-n_1))=0$, for all $m_2$, and all $(n_1, n_2)$
with $n_1-n_2$ odd, since the densities of the respective lattices in $\R^2$
are greater than 1. Then $F=f_1\otimes f_2\neq 0$ but
\[V_{\gauss_2}F(m_1,m_2,a(n_1+n_2),b(n_2-n_1))=0,\quad\forall
m_1,m_2,n_1,n_2,\] implying incompleteness of $(\gauss_2,\Lambda)$ for all $a,b>\frac12$.

We note further that
\begin{equation*}
\begin{aligned}
\Lambda&=\{(m_1,m_2,2ak_1,2bk_2)^T:m_1,m_2,k_1,k_2\in\Z\}\\
&\cup\{(m_1,m_2,2ak_1+a,2bk_2+b)^T:m_1,m_2,k_1,k_2\in\Z\}.
\end{aligned}
\end{equation*}
If $a,b=\frac12$, the system $(\gauss_2,\Lambda)$ is complete, because it
is the union of two complete systems. However, it is not a frame for
$L^2(\R^2)$: we can choose $\epsilon>0$ and
$f_1,f_2\in L^2(\R)$ with unit norm such that
\[
\sum_{k,l\in\Z}|V_{\gauss_1}f_1(k,l)|^2<\epsilon,\quad
\sum_{k,l\in\Z}|V_{\gauss_1}f_2(k,l+\tfrac12)|^2<\epsilon.
\]
Then letting $F=f_1\otimes f_2$, it is not difficult to see that
\begin{align*}
&\sum_{m_1,m_2,n_1,n_2} |V_{\gauss_2}F(m_1,m_2,\tfrac12(n_1+n_2), \tfrac12(n_2-n_1)|^2
\\ &\qquad =\sum_{m_1,m_2,\underbrace{n_1,n_2}_{2|n_1-n_2}} |V_{\gauss_2}F(m_1,m_2,\tfrac12(n_1+n_2), \tfrac12(n_2-n_1)|^2
 \\ &\qquad \qquad   +\sum_{m_1,m_2,\underbrace{n_1,n_2}_{2\nmid n_1-n_2}}|V_{\gauss_2}F(m_1,m_2,\tfrac12(n_1+n_2), \tfrac12(n_2-n_1)|^2\\
 &\qquad \le
\sum_{k,l\in\Z}|V_{\gauss_1}f_1(k,l)|^2\cdot\sum_{k,l\in\Z}|V_{\gauss_1}f_2(k,l)|^2 \\ & \qquad \qquad +
\sum_{k,l\in\Z}|V_{\gauss_1}f_1(k,l+\tfrac12)|^2\cdot\sum_{k,l\in\Z}|V_{\gauss_1}f_2(k,l+\tfrac12)|^2 \le 2C\epsilon,
\end{align*}
where $C$ is the norm of the $\ell^2(\Z^2)$-valued bounded analysis operator $D:f\mapsto\{V_{\gauss_1}f(\lambda):\lambda\in\Z^2\}$, see~\cite{G01}, Proposition 12.2.5.
This implies that no lower frame
bound exists for $(\gauss_2,\Lambda)$.

If  $a,b<\frac12$, then $(\gauss_2,\Lambda)$ is a frame for $L^2(\R^2)$,
because it is the union of two frames for $L^2(\R^2)$. \end{proof}

\emph{Remark}: Unfortunately, the cases $a>\frac12,b<\frac12$ or $a<\frac12,b>\frac12$ are not answered by Proposition~\ref{prop:gaussian-block-diagonal}.

Generalizing the ideas underlying
Proposition~\ref{prop:gaussian-block-diagonal} leads to the following
 result for lattices $\Lambda$ with a particular subgroup structure:
\begin{theorem}\label{th:thmg}
Let $\odot_{i=1}^d A_i\Z^2$ be a subgroup of $\Lambda\subset\R^{2d}$ of index
$n$. If there exist natural numbers $l_i, 1\le i\le d$, such that
$\sum_{i=1}^d l_i=n$ and $l_i<\det A_i$, then the system $(\gauss_d,\Lambda)$ is
incomplete in $\LtR$.
\end{theorem}
\begin{proof}
We split the $n$ cosets of $\odot_{i=1}^d A_i\Z^2$ into $d$ groupings
$\Delta_1,\ldots ,\Delta_d$ such that $|\Delta_i|=l_i$. $\Delta_i$ contains coset representatives denoted by $[\tau]$. We have
\[\Lambda=\bigcup_{i=1}^d\bigcup_{[\tau]\in\Delta_i}\{A_1\Z^2\times\ldots\times  A_d\Z^2\}+[\tau],\]
The short-time Fourier transform of the tensor product $\otimes_{i=1}^d f_i$
factorizes, namely
\[V_{\gauss_d}(\otimes_{i=1}^d f_i)(x,\omega)=\prod_{i=1}^d
V_{\gauss_1}f_i(x_i,\omega_i),\] where $(x_i,\omega_i)\in A_i\Z^2+[\tau_i]$, $[\tau_i]$ being the coset representative of $A_i\Z^2$ in the restriction of $\odot_{i=1}^d A_i\Z^2$ to $A_i\Z^2$.
As the density of the set \[U_i=\bigcup_{[\tau]\in\Delta_i}A_i\Z^d+[\tau_i],\quad1\le i\le d\]
is $l_iD(A_i)<1$,
Theorem~\ref{th:lyub} applies and non-zero functions $f_i\in L^2(\R)$ can be
chosen so that $V_{\gauss_1}f_i(x_i,\omega_i)=0$, for all $(x_i,\omega_i)\in
U_i$. Then as in Proposition~\ref{prop:gaussian-block-diagonal} we
conclude that $V_{\gauss_d}(\otimes_{i=1}^d f_i)$ vanishes on all of
$\Lambda$, but $\otimes_{i=1}^d f_i\neq0$. Hence, this Gabor system is
incomplete in $L^2(\R^d)$.
\end{proof}

\textit{Remark}: If $\Lambda$ satisfies the hypothesis of
Theorem~\ref{th:thmg}, then the density theorem implies incompleteness if
$D(\Lambda)=n\prod_{i=1}^d \frac 1 {\det A_i} <1$, that is, if $ \prod_{i=1}^d
\det A_i  >n$. Hence, for Theorem~\ref{th:thmg} to be relevant, we need to
combine the condition $ \prod_{i=1}^d \det A_i \leq n$ with the condition
$\det A_i > l_i$ and  $\sum_{i=1}^d l_i =n$. This leads to
\begin{eqnarray}\label{eqn:ProductSumInequality}
\prod_{i=1}^d l_i < \sum_{i=1}^d l_i.
\end{eqnarray}
Assuming without loss of generality the order $l_1\geq l_2\geq \ldots \geq l_d >0$, we divide \eqref{eqn:ProductSumInequality} by $l_1$ and observe that then
$\prod_{i=2}^d l_i < d$. As all $l_i$ are positive integers, we conclude that $l_2=l_3=l_4=\ldots= l_d=1$
and $l_1=n-d+1$. 

Note that Theorem~\ref{th:thmg} implies the incompleteness asserted in
Proposition~\ref{prop:gaussian-block-diagonal} for $a,b>\frac12$ because
$(\Z\times 2a\Z)\odot (\Z\times 2b\Z)$ is a subgroup of $\Lambda$ of index
2 and $l_1=l_2=1<2a,2b$. Similarly, we can deduce the following result.
\begin{corollary}\label{prop:gaussian-block-diagonal2} Let
$\Lambda=\left(\begin{smallmatrix}
ak& a\\
0&b
\end{smallmatrix}\right)\Z^2\times\Z^2, k\in\N$. Then the Gabor system $(\gauss_2,\Lambda)$ is
incomplete if there exists $l\in\N$ such that $a>\frac
lk,b>\frac{k-l}k$.
\end{corollary}
\begin{proof}
The subgroup $(ak\Z\times\Z)\odot(bk\Z\times\Z)$ has index $k$ in
$\Lambda$. If there exist $l_1,l_2$ such that $l_1<ak,l_2<bk$ and
$l_1+l_2=k$, the result  follows by Theorem~\ref{th:thmg}.
\end{proof}
\emph{Remark:} The range of parameters $k,l$, where the condition from
Corollary~\ref{prop:gaussian-block-diagonal2} is stronger than the density
condition ($abk>1$) is quite small if $k\ge5$: the only values of $l$ for which
\[\frac1k>ab>\frac lk\cdot\frac{k-l}k\] are $l=1,k-1$ because always
$2(k-2)>k$.

We can combine the results above to obtain further examples.
\begin{proposition}
Let $\Lambda=\Z^3\times\left(\begin{smallmatrix}
a& a&0\\
-b&b&0\\
0&0&c
\end{smallmatrix}\right)\Z^3$. Then the Gabor system $(\gauss_3,\Lambda)$ is a
frame for $L^2(\R^3)$ if $a,b<\frac12,c<1$. If $a=b\le\frac12,c=1$,
$(\gauss_3,\Lambda)$ is complete, but not a frame for $L^2(\R^3)$. If
$a,b>\frac12$ or $c>1$, $(\gauss_3,\Lambda)$ is incomplete.
\end{proposition}
\begin{proof}
We choose $F=f_1\otimes f_2\otimes f_3\in L^2(\R^3)$ in order to apply a tensor argument as~\eqref{eq:eqn-factoriz}.
When $a,b>\frac 12$, the claim follows immediately from Proposition~\ref{prop:gaussian-block-diagonal}. 
When $c>1$, it suffices to choose $f_3$ which is in
the orthogonal complement of $\{T_m M_{cn}\gauss_1:m,n\in\Z\}$ and repeat
the same line of reasoning.

Whenever $a,b<\frac12,c<1$, then $(\gauss_3,\Lambda)$ is a frame, because
it is the product of two frames (see Lemma~\ref{lemma:tensor-gabor-frames} and
Proposition~\ref{prop:gaussian-block-diagonal}).
\end{proof}
\textit{Remark}: The herein presented results extend to some (but not all)
lattices which are symplectically identical to those listed. For a symplectic
transformation $M$ with associated metaplectic operator $\mu(M)$, the
spanning properties of $(\mu(M)g,M\Lambda)$ are equivalent to those of
$(g,\Lambda)$~\cite{Fol89,G01}. Furthermore, the metaplectic operators
$\mu(M)$ associated to symplectic matrices $M$ of the form
$\left(\begin{smallmatrix}
B &0\\
0&(B^\ast)^{-1} \end{smallmatrix}\right) $,  $B$ unitary, respectively
$\left(\begin{smallmatrix}
  0&Id\\
-Id&0
 \end{smallmatrix}\right)
$, are dilation by the unitary matrix $B$, respectively the Fourier
transform~\cite{Fol89,G01}. Both leave the Gaussian invariant, hence, for
such $M$,  $(\gauss_d,M\Lambda)$ is a frame if and only if
$(\gauss_d,\Lambda)$ is.

The results presented only scratch the surface of the theory of multivariate
Gaussian Gabor frames. But we hope that together with~\cite{GroGa}, they
will motivate further study on the subject.

%

\bibliographystyle{elsart-num}
\bibliography{bibliogr}

\end{document}